\documentclass[11pt,a4paper]{article}
\usepackage{graphicx}
\usepackage{amsmath,amsfonts,amssymb,amsthm}
\title{${\rm SO}(n)$ covariant local tensor valuations on polytopes}
\author{Daniel Hug and Rolf Schneider}
\date{}
\newcommand{\Sn}{{\mathbb S}^{n-1}}
\newcommand{\R}{{\mathbb R}}
\newcommand{\Kn}{{\mathcal K}^n}
\newcommand{\Pn}{{\mathcal P}^n}
\newcommand{\Rn}{{\mathbb R}^n}
\newcommand{\N}{{\mathbb N}}

\newcommand{\cH}{\mathcal{H}}
\newcommand{\Ha}{\mathcal{H}}
\newcommand{\bS}{\mathbb{S}}
\newcommand{\T}{\mathbb{T}}
\newcommand{\B}{\mathcal{B}}
\newcommand{\D}{{\rm d}}
\newcommand{\s}{{\mathbb S}}
\newcommand{\F}{{\mathcal F}}

\newtheorem{lemma}{Lemma}
\newtheorem{theorem}{Theorem}
\newtheorem{proposition}{Proposition}

\begin{document}

\maketitle

\begin{abstract}
The Minkowski tensors are valuations on the space of convex bodies in ${\mathbb R}^n$ with values in a space of symmetric tensors, having additional covariance and continuity properties. They are extensions of the intrinsic volumes, and as these, they are the subject of classification theorems, and they admit localizations in the form of measure-valued valuations. For these local tensor valuations, restricted to convex polytopes, a classification theorem has been proved recently, under the assumption of isometry covariance, but without any continuity assumption. This characterization result is extended here, replacing the covariance under orthogonal transformations by invariance under proper rotations only. This yields additional local tensor valuations on polytopes in dimensions two and three, but not in higher dimensions. They are completely classified in this paper. 

\medskip


\medskip

\end{abstract}

\section{Introduction}\label{sec1}

A valuation on the space $\Kn$ of convex bodies in $\Rn$ is a mapping $\varphi$ from $\Kn$ into some abelian group with the property that
$$ \varphi(K\cup L) +\varphi(K\cap L)=\varphi(K)+\varphi(L)$$
whenever $K,L,K\cup L\in\Kn$. The best known examples are the intrinsic volumes or Minkowski functionals. They arise as the suitably normalized coefficients of the polynomial in $\rho$ which expresses, for a given convex body $K$, the volume of the outer parallel body of $K$ at distance $\rho\ge 0$. The celebrated characterization theorem of Hadwiger states that every rigid motion invariant continuous real-valued valuation on $\Kn$ is a linear combination of the intrinsic volumes. This theorem was the first culmination of a rich theory of valuations on convex bodies (for the older history, see the surveys \cite{McM93}, \cite{McMS83}), which in the last two decades has  again been widened and deepened considerably. For an introduction and for references, we refer to \cite{Sch14}, in particular Chapter 6 and Section 10.16. A survey on recent developments is given by Alesker \cite{Ale14}.

A natural extension of the intrinsic volumes is obtained if the volume is replaced by a higher moment. If the integral
$$ \int_K x\otimes \cdots\otimes x\,\D x,$$
where the integrand is an $r$-fold tensor product ($r\in\N$), is evaluated for the outer parallel body of $K$ at distanc $\rho$, one again obtains a polynomial in $\rho$, and its coefficients can be expressed as sums of symmetric tensors, which are functions of $K$. Suitably normalized, these yield the so-called Minkowski tensors. They are tensor-valued valuations on $\Kn$, with additional continuity and isometry covariance properties. (For brief introductions, we refer to \cite{Sch14}, Subsection 5.4.2, and to \cite{HS16}.) After some sporadic treatments (e.g., \cite{Mul53}, \cite{Prz97}), a thorough investigation of the Minkowski tensors began with the work of McMullen \cite{McM97}, who studied them on polytopes. Alesker \cite{Ale99b} (based on his work in \cite{Ale99a}) extended Hadwiger's classification theorem, showing that the real vector space of continuous, isometry covariant tensor valuations on $\Kn$ of given rank is spanned by suitable Minkowski tensors, multiplied by powers of the metric tensor. Questions of linear independence, leading to the determination of dimensions and bases, were treated in \cite{HSS08a}. Minkowski tensors were studied and used in integral geometry (\cite{BH14}, \cite{HSS08b}, \cite{Sch00}, \cite{SS06}), in stochastic geometry, stereology and image analysis (\cite{HHKM14}, \cite{HKS15}, \cite{KKH14}, \cite{SS02}), and for lower dimensions and ranks they were applied in physics (\cite{BDMW02}, \cite{MKSM13}, \cite{SchT10}, \cite{SchT11}, \cite{SchT12}). We also refer the reader to the Lecture Notes \cite{KJ16}.

Just as the intrinsic volumes have local versions, the support measures, with curvature and area measures as marginal measures, so the Minkowski tensors have local versions. They associate with every convex body a series of tensor-valued measures. The mappings defined in this way are valuations, with the additional properties of weak continuity, isometry covariance, and local determination. A corresponding classification theorem was proved in \cite{HS14}, based on the previous investigation \cite{Sch13} concerning the case of polytopes. This approach has some interesting features. First, on polytopes, a complete classification is possible without any continuity assumption. Which of the obtained local tensor valuation mappings have weakly continuous extensions to all convex bodies, was determined in \cite{HS14}. Second, the valuation property need not be assumed, but is a consequence.

The isometry covariance that is assumed in these characterization results has two components: translation covariance (a certain polynomial behaviour under translations), and covariance with respect to the orthogonal group ${\rm O}(n)$. Covariance with respect to other groups is also of interest. Recently, Haberl and Parapatits \cite{HP15} were able to classify all measurable ${\rm SL}(n)$ covariant symmetric tensor valuations on convex polytopes containing the origin in the interior. In the opposite direction (a smaller group than ${\rm O}(n)$), it was shown by Saienko \cite{Sai16}, under continuity and smoothness assumptions, that the classification of the local tensor valuations does not change for $n\ge 4$ if ${\rm O}(n)$ covariance is replaced by ${\rm SO}(n)$ covariance. In the physically relevant dimensions two and three, however, he surprisingly discovered additional local tensor valuations. It is the purpose of this paper to study ${\rm SO}(n)$ covariant local tensor valuations on polytopes, without assuming any continuity property, and also obtaining the valuation property as a consequence. Thus, the aim is to extend the results of \cite{Sch13}, replacing the orthogonal group ${\rm O}(n)$ by the group ${\rm SO}(n)$ of proper rotations. The main result is Theorem \ref{Thm2} below. Which of the newly found mappings have a weakly continuous extension to all convex bodies, will be studied elsewhere.

After collecting some notation in Section \ref{sec2}, we formulate our results in Section \ref{sec3}. The proof is prepared by some auxiliary results in Section \ref{sec4} and the refinement of two lemmas from \cite{Sch13} in Section \ref{sec5}. The main result is then proved in Section \ref{sec6}.

\section{Notation}\label{sec2}

We work in $n$-dimensional Euclidean space $\R^n$ ($n\ge 2$), with scalar product $\langle  \cdot \,, \cdot \rangle$ and induced norm $\|\cdot\|$. Its unit sphere is $\Sn$, and we write $\Sigma^n:= \R^n \times \Sn$ and equip this with the product topology. By $G(n,k)$ we denote the Grassmannian of $k$-dimensional linear subspaces of $\R^n$, $k\in\{0,\dots,n \}$. If $L\in G(n,k)$, we write $\s_L:=\s^{n-1}\cap L$. The orthogonal complement of $L\in G(n,k)$ is denoted by $L^\perp$. By $\Ha^k$ we denote the $k$-dimensional Hausdorff measure on $\R^n$, and $\Ha^{n-1}(\Sn)$ defines the  constant $\omega_n=2\pi^{n/2}/\Gamma(n/2)$. If $S$ is a topological space, we write $\B(S)$ for the $\sigma$-algebra of its Borel sets. For $S\subset\R^n$, the set of bounded Borel sets in $S$ is denoted by $\B_b(S)$.

The {\em orthogonal group} ${\rm O}(n)$ of $\R^n$ is the group of all linear mappings of $\R^n$ into itself preserving the scalar product, and  ${\rm SO}(n)$ is the subgroup of {\em rotations}, which preserve also the orientation.

By $\Pn$ we denote the set of (convex and nonempty) polytopes in $\R^n$. For $k\in\{0,\dots,n\}$, the set of $k$-dimensional faces of the polytope $P$ is denoted by $\F_k(P)$. For $F\in \F_k(P)$, the subspace $L(F)\in G(n,k)$, the {\em direction space} of $F$, is the translate, passing through $0$, of the affine hull of $F$. The set $\nu(P,F)\subset \s_{L(F)^\perp}$ is the set of outer unit normal vectors of $P$ at its face $F$. The generalized normal bundle (or normal cycle) of $P$ is the subset ${\rm Nor}\,P\subset\Sigma^n$ consisting of all pairs $(x,u)$ such that $x$ is a boundary point of $P$ and $u$ is an outer unit normal vector of $P$ at $x$.

This paper rests heavily on the previous papers \cite{Sch13} and \cite{HS14} on local tensor valuations and uses much of their terminology. We recall here briefly the underlying conventions on tensors. For $p\in\N_0$, we denote by $\T^p$ the real vector space of symmetric tensors of rank $p$ (or symmetric $p$-tensors, for short) on $\R^n$. The scalar product $\langle\cdot\,,\cdot\rangle$ of $\R^n$ is used to identify $\R^n$ with its dual space, so that each vector $a\in\R^n$ is identified with the linear functional $x\mapsto \langle a,x\rangle$, $x\in\R^n$. Thus, $\T^1$ is identified with $\R^n$ (and $\T^0$ with $\R$), and for $p\ge 1$, each tensor $T\in\T^p$ is a symmetric $p$-linear functional on $\R^n$. The symmetric tensor product $a\odot b$ is always abbreviated by $ab$, and for $x\in\R^n$, the $r$-fold symmetric tensor product $x\odot\dots\odot x$ is denoted by $x^r$. 

The {\em metric tensor} $Q$ on $\R^n$ is defined by $Q(x,y):=\langle x,y\rangle$ for $x,y\in \R^n$. For a subspace $L\in G(n,k)$, we denote by $\T^p(L)$ the space of symmetric $p$-tensors on $L$. We must distinguish between $Q_{(L)}$, the metric tensor on $L$, with $Q_{(L)}(a,b):=\langle a,b\rangle$ for $a,b\in L$, and the tensor $Q_{L}$, defined by 
$$ Q_{L}(a,b):= \langle \pi_L a,\pi_L b\rangle \quad\mbox{for }a,b\in\R^n,$$
where $\pi_L:\R^n\to L$ denotes the orthogonal projection. The mapping $\pi_L^*: \bigcup_{p\in\N_0} \T^p(L)\to  \bigcup_{p\in\N_0} \T^p$ is defined by $(\pi_L^* T)(a_1,\dots,a_p):= T(\pi_L a_1,\dots,\pi_L a_p)$, $a_1,\dots,a_p\in \R^n$, for $T\in\T^p(L)$. In particular, $\pi_L^* Q_{(L)}=Q_L$.  (This notation is different from the one used in \cite{Sch13}.)

\section{Formulation of Results}\label{sec3}

The {\em Minkowski tensors} of a convex body $K\in\Kn$ are given by
$$  
\Phi_k^{r,s}(K) = \frac{1}{r!s!} \frac{\omega_{n-k}}{\omega_{n-k+s}} \int_{\Sigma^n} x^ru^s\,\Lambda_k(K,\D(x,u))
$$
for $k=0,\dots,n-1$ and $r,s\in\N_0$. We refer to \cite{Sch14}, Section 4.2, for the support measures $\Lambda_0(K,\cdot),\dots,\Lambda_{n-1}(K,\cdot)$ appearing here, and to \cite{Sch14}, Section 5.4, for a brief introduction to the Minkowski tensors. The {\em local Minkowski tensors} are defined by
$$  
\phi_k^{r,s}(K,\eta) = \frac{1}{r!s!} \frac{\omega_{n-k}}{\omega_{n-k+s}} \int_\eta x^ru^s\,\Lambda_k(K,\D(x,u))
$$
for $\eta\in\B(\Sigma^n)$. If $P\in\Pn$ is a polytope, the special form of the support measures yields a more explicit expression, namely
$$ \phi_k^{r,s}(P,\eta) = C_{n,k}^{r,s} \sum_{F\in\F_k(P)}\int_F \int_{\nu(P,F)} {\bf 1}_\eta(x,u) x^r u^s\, \Ha^{n-k-1}(\D u)\,\Ha^k(\D x),$$
where we now use the abbreviation
$$ C_{n,k}^{r,s}:= (r!s!\omega_{n-k+s})^{-1}$$
and where the function ${\bf 1}_\eta$ is the characteristic function of $\eta$. The attempt to characterize these local tensor valuations on polytopes by their basic properties revealed in \cite{Sch13} that these properties are also shared by the {\em generalized local Minkowski tensors}. For a polytope $P\in\Pn$ these are defined by
\begin{equation}\label{2.6}  
\phi_k^{r,s,j}(P,\eta) := C_{n,k}^{r,s} \sum_{F\in\F_k(P)} Q_{L(F)}^{\hspace*{1pt}j}\int_F \int_{\nu(P,F)} {\bf 1}_\eta(x,u) x^r u^s\, \Ha^{n-k-1}(\D u)\,\Ha^k(\D x)
\end{equation}
for $\eta\in\B(\Sigma^n)$, $k\in\{0,\dots, n-1\}$, $r,s\in\N_0$, and for $j\in{\mathbb N}_0$ if $k>0$, but only $j=0$ if $k=0$. Recall that $x^r u^s$ in (\ref{2.6}) denotes a symmetric tensor product, and that also the product of  $Q_{L(F)}^{\hspace*{1pt}j}$ with the subsequent tensor-valued integral is a symmetric tensor product. 

For fixed $k,r,s,j$ and with $p:=2j+r+s$, the tensor $\phi_k^{r,s,j}$ defines a mapping $\Gamma:\Pn \times \B(\Sigma^n)\to\T^p$. For such a mapping $\Gamma$, the following properties are of interest. $\Gamma$ is called {\em translation covariant} of degree $q\le p$ if
\begin{equation}\label{tc} 
\Gamma(P+t,\eta+t) =\sum_{j=0}^q \Gamma_{p-j}(P,\eta)\frac{t^j}{j!}
\end{equation}
with tensors $ \Gamma_{p-j}(P,\eta)\in\T^{p-j}$, for $P\in\Pn$, $\eta\in\B(\Sigma^n)$, and $t\in\R^n$. Here $\eta+t:= \{(x+t,u): (x,u)\in\eta\}$, and $\Gamma_p=\Gamma$. If $\Gamma$ is translation covariant of degree zero, it is called {\em translation invariant}, and $\Gamma$ is just called {\em translation covariant} if it is translation covariant of some degree $q\le p$. The mapping $\Gamma$ is called ${\rm SO}(n)$ {\em covariant} if $\Gamma(\vartheta P,\vartheta \eta)= \vartheta \Gamma(P,\eta)$ for $P\in\Pn$, $\eta\in\B(\Sigma^n)$, $\vartheta\in{\rm SO}(n)$, where $\vartheta\eta:= \{(\vartheta x,\vartheta u): (x,u)\in\eta\}$. Here the operation of  ${\rm SO}(n)$ on $\T^p$ is defined by $(\vartheta T)(x_1,\dots,x_p):= T(\vartheta^{-1}x_1,\dots,\vartheta^{-1}x_p)$ for $x_1,\dots,x_p\in\R^n$ and $\vartheta\in {\rm SO}(n)$. Similarly, ${\rm O}(n)$ covariance is defined. Finally, the mapping $\Gamma$ is {\em locally defined} if $\eta\cap {\rm Nor}\,P = \eta'\cap {\rm Nor}\,P'$ with $P,P'\in\Pn$ and $\eta,\eta'\in\B(\Sigma^n)$ implies $\Gamma(P,\eta)=\Gamma(P',\eta')$.

The mapping defined by $\Gamma(P,\eta):= \phi_k^{r,s,j}(P,\eta)$, for fixed $k,r,s,j$, has the following properties. For each $P\in\Pn$, $\Gamma(P,\cdot)$ is a $\T^p$-valued measure, with $p=2j+r+s$. $\Gamma$ is translation covariant, ${\rm O}(n)$ covariant, and locally defined. These properties are not changed (except that the rank must be adjusted) if $\Gamma$ is multiplied (symmetrically) by a power of the metric tensor.

The following theorem was essentially proved in \cite{Sch13}, with some simplifications and supplements provided in \cite{HS14}.

\begin{theorem}\label{Thm1}
For $p\in\N_0$, let $T_p(\Pn)$ denote the real vector space of all mappings $\Gamma:\Pn\times\B(\Sigma^n)\to\T^p$ with the following properties.\\[1mm]
$\rm (a)$ $\Gamma(P,\cdot)$ is a $\T^p$-valued measure, for each $P\in\Pn$,\\[1mm]
$\rm (b)$ $\Gamma$ is translation covariant and ${\rm O}(n)$ covariant,\\[1mm]
$\rm (c)$ $\Gamma$ is locally defined.\\[1mm]
Then a basis of $T_p(\Pn)$ is given by the mappings $Q^m\phi^{r,s,j}_k$, where $m,r,s,j\in\N_0$ satisfy $2m+2j+r+s=p$, where $k\in\{0,\dots,n-1\}$, and where $j=0$ if $k\in\{0,n-1\}$.
\end{theorem}

The purpose of the following is to extend this characterization of local tensor valuations on polytopes from ${\rm O}(n)$ to ${\rm SO}(n)$ covariance. It was discovered by Saienko \cite{Sai16} that under this weaker assumption there are additional tensor valuations in dimensions two and three. 

In the following, the spaces $\R^2$ and $\R^3$ are endowed with fixed orientations. Let $P\in{\mathcal P}^3$. For each edge $F\in{\mathcal F}_1(P)$, we choose a unit vector $v_F\in L(F)$. For $u\in \s_{L(F)^\perp}$, let $v_F\times u =: \overline u$ denote the vector product of $v_F$ and $u$ in ${\mathbb R}^3$; thus $\overline u$ is the unique unit vector such that $(v_F,u, \overline u)$ is a positively oriented orthonormal basis of $\R^3$. We define
\begin{align}\label{1a} 
\widetilde\phi^{r,s,j}(P,\eta) &:= \sum_{F\in{\mathcal F}_1(P)} Q^j_{L(F)} v_F\int_F\int_{\nu(P,F)} {\bf 1}_\eta(x,u)x^r(v_F\times u) u^s\,\cH^1(\D u)\, \cH^1(\D x)\\
 &= \sum_{F\in{\mathcal F}_1(P)}  v_F^{2j+1}\int_F\int_{\nu(P,F)} {\bf 1}_\eta(x,u)x^r(v_F\times u) u^s\,\cH^1(\D u)\, \cH^1(\D x) \label{1c}
\end{align}
for $\eta\in\B(\Sigma^3)$ and $r,s,j\in{\mathbb N}_0$. Here we have used that $Q_{L(F)}=v_F^2$, since $\dim F=1$. The tensor $\widetilde\phi^{r,s,j}(P,\eta)$ is well-defined, since it does not change if the  vector $v_F$ is replaced by $-v_F$. Since
\begin{equation}\label{01} 
\widetilde\phi^{r,s,j}(P+t,\eta+t) =\sum_{i=0}^r \binom{r}{i}\widetilde\phi^{r-i,s,j}(P,\eta)\,t^i
\end{equation}
for $t\in{\mathbb R}^3$, the mapping $\widetilde\phi^{r,s,j}$ is translation covariant. It is also ${\rm SO}(3)$ covariant, since $\vartheta v_F\times\vartheta u=\vartheta(v_F\times u)$ for $\vartheta\in{\rm SO}(3)$.

Now let $n=2$. For $u\in {\mathbb S}^1$, let $\overline u\in {\mathbb S}^1$ be the unique vector for which $(u,\overline u)$ is a positively oriented orthonormal basis of $\R^2$. For $P\in{\mathcal P}^2$, $k\in\{0,1\}$ and $\eta \in \B(\Sigma^2)$ we define
\begin{equation}\label{1b} 
\widetilde \phi_k^{r,s}(P,\eta):= \sum_{F\in\F_k(P)} \int_F \int_{\nu(P,F)} {\bf 1}_\eta(x,u) x^r \overline u\, u^s\,\Ha^{1-k}(\D u)\,\Ha^k(\D x).
\end{equation}
Of course, if $\dim P=2$ and $F\in\F_1(P)$, then $\nu(P,F)=\{u_F\}$ with a unique vector $u_F$, and we have
$$ \widetilde \phi_1^{r,s}(P,\eta)= \sum_{F\in\F_1(P)}  \overline {u_F}\, u_F^s \int_F  {\bf 1}_\eta(x,u_F) x^r\,\Ha^1(\D x).$$
If $\dim P=1$ and $F\in\F_1(P)$, then $P=F$ and $\nu(P,F)=\{\pm u_F\}$, and therefore
$$ \widetilde \phi_1^{r,s}(P,\eta)=  \int_F  [{\bf 1}_\eta(x,u_F) \overline{u_F}\,u_F^s + {\bf 1}_\eta(x,-u_F) (-\overline{u_F})(-u_F)^s] x^r\,\Ha^1(\D x).$$
For the case $k=0$ we note that for $F\in\F_0(P)$ we have $F=\{x_F\}$ and hence
$$ \widetilde \phi_0^{r,s}(P,\eta)=\sum_{F\in\F_0(P)} x_F^r\int_{\nu(P,F)} {\bf 1}_\eta(x_F,u)\overline u u^s\,\Ha^1(\D u).$$
The translation covariance and ${\rm SO}(2)$ covariance of $\widetilde \phi_k^{r,s}$ are easy to check.

The mappings $\widetilde\phi^{r,s,j}(\cdot,\eta)$ and   $\widetilde \phi_k^{r,s}(\cdot,\eta)$ ($k=0,1$), defined on polytopes in $\R^3$, respectively $\R^2$, are valuations. This is proved as it was done for the mappings $\phi_k^{r,s,j}(\cdot,\eta)$ in \cite[Theorem 3.3]{HS14}.

The following result is the counterpart to Theorem \ref{Thm1}, with the rotation group ${\rm SO}(n)$ instead of the orthogonal group ${\rm O}(n)$.

\begin{theorem}\label{Thm2}
For $p\in\N_0$, let $\widetilde T_p(\Pn)$ denote the real vector space of all mappings $\Gamma:\Pn\times\B(\Sigma^n)\to\T^p$ with the following properties.\\[1mm]
$\rm (a)$ $\Gamma(P,\cdot)$ is a $\T^p$-valued measure, for each $P\in\Pn$,\\[1mm]
$\rm (b)$ $\Gamma$ is translation covariant and ${\rm SO}(n)$ covariant,\\[1mm]
$\rm (c)$ $\Gamma$ is locally defined.

Then a basis of $\widetilde T_p(\Pn)$ is given by the mappings $Q^m\phi^{r,s,j}_k$, where $m,r,s,j\in\N_0$ satisfy $2m+2j+r+s=p$, where $k\in\{0,\dots,n-1\}$, and where $j=0$ if $k\in\{0,n-1\}$, together with\\[1mm]
$\bullet$ if $n\ge 4$, no more mappings,\\[1mm]
$\bullet$ if $n=3$, the mappings $Q^m\widetilde \phi^{r,s,j}$, where $m,r, s,j\in\N_0$ satisfy $2m+2j+r+s+2=p$,\\[1mm]
$\bullet$ if $n=2$, the mappings $Q^m\widetilde \phi_k^{r,s}$, where $m,r,s\in\N_0$ satisfy $2m+r+s+1=p$ and where \hspace*{6pt}  $k\in\{0,1\}$.
\end{theorem}

As in \cite{HS14} (and similarly earlier in Alesker's work \cite{Ale99b}), this general result follows from its special case where $\Gamma$ is translation invariant. Therefore, we formulate this special case separately, deleting the assertion of linear independence, which we discuss in the next section.

\begin{theorem}\label{Thm3}
Let $p\in\N_0$. Let $\Gamma:{\mathcal P}^n\times{\mathcal B}(\Sigma^n) \to\T^p$ be a mapping with the following properties.\\[1mm]
$\rm (a)$ $\Gamma(P,\cdot)$ is a $\T^p$-valued measure, for each $P\in\Pn$,\\[1mm]
$\rm (b)$ $\Gamma$ is translation invariant and ${\rm SO}(n)$ covariant,\\[1mm]
$\rm (c)$ $\Gamma$ is locally defined.

Then $\Gamma$ is a linear combination, with constant coefficients, of the mappings $Q^m\phi_k^{0,s,j}$, where $m,s,j\in \N_0$ satisfy $2m+2j+s=p$, where $k\in\{0,\dots,n-1\}$, and where $j=0$ if $k\in\{0,n-1\}$, together with\\[1mm]
$\bullet$ if $n\ge 4$, no more mappings,\\[1mm]
$\bullet$ if $n=3$, the mappings $Q^m\widetilde \phi^{0,s,j}$, where $m,s,j\in\N_0$ satisfy $2m+2j+s+2=p$,\\[1mm]
$\bullet$ if $n=2$, the mappings $Q^m\widetilde \phi_k^{0,s}$, where $m,s\in\N_0$ satisfy $2m+s+1=p$ and where $k\in\{0,1\}$.
\end{theorem}

In the next section, we prove the linear independence result contained in Theorem \ref{Thm2} and show how Theorem \ref{Thm2} follows from Theorem \ref{Thm3} (and Proposition \ref{Lem1}). In Section \ref{sec5}, we extend two lemmas of \cite{Sch13} from ${\rm O}(n)$ covariance to ${\rm SO}(n)$ covariance. The proof of Theorem \ref{Thm3} then follows in Section \ref{sec6}.

\section{Auxiliary Results}\label{sec4}

First we explain how Theorem \ref{Thm2} is deduced from Theorem \ref{Thm3} and Proposition \ref{Lem1} below. Each of $\phi_k^{r,s,j}$ if $n\ge 2$, each of $\widetilde \phi^{r,s,j}$ if $n=3$, and each of $\widetilde \phi_k^{r,s}$ if $n=2$, is a mapping $\Gamma: \Pn\times\B(\Sigma^n)\to \T^p$ (for suitable $p$) which has the following properties:\\[1mm]
(a) $\Gamma(P,\cdot)$ is a $\T^p$-valued measure, for each $P\in\Pn$, \\[1mm]
(b) $\Gamma$ is translation covariant of some degree $q\le p$ and ${\rm SO}(n)$ covariant, \\[1mm]
(c) $\Gamma$ is locally defined.\\[1mm] 
It follows from \cite[Lemmas 3.1, 3.2]{HS14} that each $\Gamma_{p-j}$ appearing in (\ref{tc}) satisfies
$$ \Gamma_{p-j}(P+t,\eta+t) = \sum_{r=0}^{q-j} \Gamma_{p-j-r}(P,\eta)\frac{t^r}{r!}$$
for $j=0,\dots,q$ and that $\Gamma_{p-j}$ has again the properties (a), (b), (c). In particular, the choice $j=q$ yields that $\Gamma_{p-q}$ is translation invariant. It is now clear that the procedure described in \cite[pp. 1534--1535]{HS14} allows us to deduce Theorem \ref{Thm2} from Theorem \ref{Thm3} (and Proposition \ref{Lem1}).

We turn to linear independence. 

\begin{proposition}\label{Lem1}
Let $p\in\N_0$. The local tensor valuations $Q^m\phi_k^{r,s,j}$ with $m,r,s,j\in\N_0$, $2m+2j+r+s=p$, $k\in\{0,\dots,n-1\}$ and $j=0$ if $k\in\{0,n-1\}$, together with\\[1mm]
$\bullet$ if $n=3$, the local tensor valuations $Q^m\widetilde\phi^{r,s,j}$ with $m,r,s,j\in\N_0$, $2m+2j+r+s+2=p$,\\[1mm]
$\bullet$ if $n=2$, the local tensor valuations $Q^m\widetilde\phi_k^{r,s}$ with $m,r,s\in\N_0$, $2m+r+s+1=p$, $k\in\{0,1\}$,\\[1mm]
are linearly independent.
\end{proposition}

\begin{proof}
For $n\ge 4$, the assertion is covered by \cite[Thm. 3.1]{HS14}. For the remaining cases, we extend the proof of that theorem. Let $n\in\{2,3\}$.

Let $F\in{\mathcal P}^n$ be a $d$-dimensional polytope, $d\in\{0,\dots,n-1\}$, and consider sets of the form $\eta=\beta\times \omega$ with Borel sets $\beta\subset \text{relint }F$ and $\omega\subset \s_{L(F)^\perp}$. For $k=d$, the representation (\ref{2.6}) reduces to
\begin{equation}\label{1} 
\phi_d^{r,s,j}(F,\beta\times\omega) =  Q_{L(F)}^j C_{n,d}^{r,s} \int_\beta x^r\,{\mathcal H}^d(\D x) \int_\omega u^s \, {\mathcal H}^{n-d-1}(\D u). 
\end{equation}
If $n=3$ and $d=1$, expression (\ref{1c}) yields
\begin{equation}\label{1n} 
\widetilde\phi^{r,s,j}(F,\beta\times\omega) = v_F^{2j+1} \int_\beta x^r\,{\mathcal H}^1(\D x) \int_\omega (v_F\times u)u^s \, {\mathcal H}^1(\D u). 
\end{equation}

\vspace{2mm}

\noindent{\bf Case 1:} $n=3$. Suppose that
\begin{equation}\label{7} 
\sum_{m,r,s,j,k \atop 2m+2j+r+s=p} a_{kmrsj} Q^m\phi_k^{r,s,j} + \sum_{m,r,s,j \atop 2m+2j+r+s+2=p} b_{mrsj} Q^m\widetilde\phi^{r,s,j}=0
\end{equation}
with $a_{kmrsj}, b_{mrsj}\in{\mathbb R}$ and with $a_{0mrsj}=a_{2mrsj}=0$ for $j\not=0$.

\vspace{2mm}

\noindent{\bf Subcase 1a:} $d\in\{0,2\}$. Then $\phi_k^{r,s,j}(F, \beta \times \omega) =0$ for $k\not=d$ and
$\widetilde \phi^{r,s,j}(F,\beta\times\omega)=0$ (by the choice of $\beta$ and $\omega$). It follows from (\ref{7}) and (\ref{1}) that
$$ \sum_{m,r,s,j \atop 2m+2j+r+s=p} a_{dmrsj}Q^m Q_{L(F)}^j C_{3,d}^{r,s}\int_\beta x^r\,{\mathcal H}^d(\D x) \int_\omega u^s\,{\mathcal H}^{3-d-1}(\D u)=0$$
for all $F,\beta,\omega$ as specified. The proof of \cite[Thm. 3.1]{HS14} shows that this implies that all coefficients $a_{dmrsj}$ are zero.

\vspace{2mm}

\noindent{\bf Subcase 1b:} $d=1$. Then $\phi_{k}^{r,s,j}(F,\beta\times \omega) =0$ for $k\not=1$, and from  (\ref{7}) (with $a_{kmrsj}=0$ for $k\in\{0,2\}$), (\ref{1}) and (\ref{1n}) we obtain
\begin{eqnarray*}
&& \sum_{m,r,s,j \atop 2m+2j+r+s=p} a_{1mrsj}Q^m Q_{L(F)}^j C_{3,1}^{r,s}\int_\beta x^r\,{\mathcal H}^1(\D x) \int_\omega u^s\,{\mathcal H}^{1}(\D u)\\
&& + \sum_{m,r,s,j \atop 2m+2j+r+s+2=p} b_{mrsj} Q^m v_F^{2j+1} \int_\beta x^r\,{\mathcal H}^1(\D x) \int_\omega (v_F\times u)u^s\,{\mathcal H}^{1}(\D u)=0.
\end{eqnarray*}
Since this holds for all $F,\beta,\omega$ as specified, we can argue as in the proof of \cite[Thm. 3.1]{HS14} and conclude that for each fixed $r$ and with $\overline a_{mrsj}:= a_{1mrsj}C_{3,1}^{r,s}$ we have
\begin{equation}\label{8}  
\sum_{m,s,j \atop 2m+2j+s=p-r} \overline a_{mrsj}Q^m Q_{L(F)}^j u^s + \sum_{m,s,j \atop 2m+2j+s+2=p-r} b_{mrsj} Q^m  v_F^{2j+1} (v_F\times u)u^s=0
\end{equation}
for all $u\in \s_{L(F)^\perp}$.

Let $(e_1,e_2,e_3)$ be a positively oriented orthonormal basis of $\R^3$ such that $e_1=v_F$. We apply (\ref{8}) to the $(p-r)$-tuple
$$ (\underbrace{x,\dots, x}_{p-r}) \quad\mbox{with}\quad x=x_1e_1+x_2e_2+x_3e_3 \in\R^3.$$
This gives
\begin{eqnarray*}
&& \sum_{m,s,j \atop 2m+2j+s=p-r} \overline a_{mrsj}(x_1^2+x_2^2+x_3^2)^mx_1^{2j}(u_2x_2+u_3x_3)^s\\ 
&& + \sum_{m,s,j \atop 2m+2j+s+2=p-r} b_{mrsj}(x_1^2+x_2^2+x_3^2)^mx_1^{2j+1}(-u_3x_2+u_2x_3)(u_2x_2+u_3x_3)^s=0
\end{eqnarray*}
for all $u_2,u_3\in\R$ such that $u_2e_2+u_3e_3\in{\mathbb S}^2$. Since the first sum defines a polynomial that is of even degree in $x_1$, whereas the second sum is of odd degree in $x_1$, either sum is zero. Denoting by $\theta$ the angle from $u_2e_2+u_3e_3$ to $x_2e_2+x_3e_3$, we can write the vanishing of the first sum as 
$$ 
\sum_{s\ge 0}\alpha_s\left(x_2^2+x_3^2\right)^{\frac{s}{2}}\cos^s\theta=0\quad \mbox{with} \quad \alpha_s=\sum_{m,j\atop 2m+2j=p-r-s} \overline a_{mrsj}(x_1^2+x_2^2+x_3^2)^mx_1^{2j}.
$$
Since this holds for all $\theta\in\R$, it follows that $\alpha_s=0$ for all $s$.
Similarly we obtain that 
$$ \sum_{m,j\atop 2m+2j+2=p-r-s} b_{mrsj}(x_1^2+x_2^2+x_3^2)^mx_1^{2j+1}=0 $$
for all $s$. Now the proof of \cite[Thm. 3.1]{HS14} shows that all coefficients $\overline a_{mrsj}, b_{mrsj}$ are zero.

\vspace{2mm}

\noindent{\bf Case 2:} $n=2$. Then $\phi_k^{r,s,j}\not=0$ only for $k\in\{0,1\}$ and hence also only for $j=0$. Therefore, we assume that
\begin{equation}\label{10} 
\sum_{m,r,s,k \atop 2m+r+s=p} a_{kmrs} Q^m\phi_k^{r,s,0} + \sum_{m,r,s,k \atop 2m+r+s+1=p} b_{kmrs} Q^m\widetilde\phi_k^{r,s}=0
\end{equation}
with $a_{kmrs}, b_{kmrs}\in{\mathbb R}$ and $k\in\{0,1\}$.
\vspace{2mm}

\noindent{\bf Subcase 2a:} $d=0$. Then $\phi_1^{r,s,0}(F,\beta\times\omega)=0$ and $\widetilde\phi_1^{r,s}(F,\beta\times \omega)=0$. 
Writing $\overline a_{0mrs}=a_{0mrs}\,C_{2,0}^{r,s}$, it follows from (\ref{10}) that 
\begin{eqnarray*}
&&  \sum_{m,r,s \atop 2m+r+s=p} \overline a_{0mrs} \, Q^m  \int_\beta x^r\,\Ha^0(\D x) \int_\omega u^s\,\Ha^1(\D u) \\ 
&& + \sum_{m,r,s \atop 2m+r+s+1=p} b_{0mrs} Q^m \int_\beta x^r\,\Ha^0(\D x)\int_\omega \overline u u^s\,\Ha^1(\D u)=0.
\end{eqnarray*}
Since this holds for all $F,\beta,\omega$ as specified, we obtain for each fixed $r$ that
$$ \sum_{m,s\atop 2m+s=p-r} \overline a_{0mrs} \,Q^mu^s+\sum_{m,s\atop 2m+s+1=p-r} b_{0mrs}Q^m\overline u u^s=0$$
for all $u\in \s^1$. We now choose the orthogonal basis $(e_1,e_2)$ of $\R^2$ such that $e_1=u$ and 
$e_2=\overline{u}$. Applying (\ref{10}) to the $(p-r)$-tuple
$$ (\underbrace{x,\dots, x}_{p-r}) \quad\mbox{with}\quad x=x_1e_1+x_2e_2\in\R^2,$$
we obtain
$$ \sum_{m,s \atop 2m+s=p-r} \overline a_{0mrs}(x_1^2+x_2^2)^m x_1^s +\sum_{m,s \atop 2m+s+1=p-r} 
b_{0mrs}(x_1^2+x_2^2)^mx_2x_1^s=0.$$
Since the first summand contains only even powers of $x_2$ and the second summand only odd powers, either summand must be zero, hence we get
$$ \sum_{m=0}^{\lfloor(p-r)/2\rfloor} c_m(x_1^2+x_2^2)^m x_1^{p-r-2m}=0,\qquad \sum_{m=0}^{\lfloor(p-r-1)/2\rfloor}d_m(x_1^2+x_2^2)^m x_2x_1^{p-r-2m-1}=0$$
with $c_m=\overline a_{0mr(p-r-2m)}$ and $d_m= b_{0mr(p-r-2m-1)}$. This yields that all coefficients $c_m,d_m$ are zero, hence all coefficents in (\ref{10}) are zero.

\vspace{2mm}

\noindent{\bf Subcase 2b:} $d=1$. Then $\phi_0^{r,s,0}(F,\beta\times\omega)=0$ and $\widetilde\phi_0^{r,s}(F,\beta\times \omega)=0$. We choose $\omega=\{u_F\}$, where $u_F$ is one of the two unit normal vectors of $F$. 
Then from (\ref{10}), applied to $(F,\beta\times\omega)$, we obtain
$$ \sum_{m,r,s \atop 2m+r+s=p} a_{1mrs} Q^m  C_{2,1}^{r,s}\, u_F^s \int_\beta x^r\,\Ha^1(\D x)+ 
\sum_{m,r,s \atop 2m+r+s+1=p} b_{1mrs} Q^m\, \overline{u_F}\,u_F^s \int_\beta x^r\,\Ha^1(\D x)=0.$$
As above, for each fixed $r$ and with $\overline a_{1mrs}:= a_{1mrs}C_{2,1}^{r,s}$ this yields
\begin{equation}\label{11}  
\sum_{m,s \atop 2m+s=p-r} \overline a_{1mrs}Q^mu_F^s + \sum_{m,s \atop 2m+s+1=p-r} b_{1mrs} Q^m \,\overline{u_F}\,u_F^s=0.
\end{equation}
We choose the orthogonal basis $(e_1,e_2)$ of $\R^2$ such that $e_1=u_F$ and $e_2=\overline{u_F}$. Applying (\ref{11}) to the $(p-r)$-tuple
$$ (\underbrace{x,\dots, x}_{p-r}) \quad\mbox{with}\quad x=x_1e_1+x_2e_2\in\R^2,$$
we obtain
$$ \sum_{m,s \atop 2m+s=p-r} \overline a_{1mrs}(x_1^2+x_2^2)^m x_1^s +\sum_{m,s \atop 2m+s+1=p-r} b_{1mrs}(x_1^2+x_2^2)^mx_2x_1^s=0.$$
Now we can conclude as before that all coefficients in (\ref{11}) and hence in \eqref{10} are zero.
\end{proof}

\section{Some Refined Lemmas}\label{sec5}

In this section, we extend Lemmas 3 and 4 in \cite{Sch13}, essentially from ${\rm O}(n)$ covariance to ${\rm SO}(n)$ covariance. (We remark that in Lemma 3 of \cite{Sch13}, the group ${\rm SO}(n)$ should be replaced by ${\rm O}(n)$, since this is used in the proof. This does not affect the rest of the paper, where Lemma 3 is only applied with ${\rm O}(n)$.)

\begin{lemma}\label{L3} 
Let $L\in G(n,k)$ with $k\in\{1,\dots,n-1\}$, let $r\in\N_0$ and $T\in\T^{\hspace*{1pt}r}$.

\noindent $\rm (a)$ Let $k\ge 2$. If $\vartheta T=T$ for each $\vartheta\in{\rm SO}(n)$ that fixes $L^\perp$ pointwise, then 
\begin{equation}\label{00}  
T= \sum_{j=0}^{\lfloor r/2\rfloor} Q_L^j\pi_{L^\perp}^*T^{(r-2j)} 
\end{equation}
with tensors $T^{(r-2j)}\in \T^{\hspace*{1pt}r-2j}(L^\perp)$, $j=0,\dots,\lfloor r/2\rfloor$.

\noindent $\rm (b)$ Let $k=1$. Let $v_L$ be a unit vector spanning $L$. Then 
$$ T = \sum_{j=0}^r v_L^j\pi_{L^\perp}^*T^{(r-j)} $$
with tensors $T^{(m)}\in \T^{\hspace*{1pt}m}(L^\perp)$.
\end{lemma}

\begin{proof}  
Given an orthonormal basis $(e_1,\dots,e_n)$ of $\Rn$, we associate with $T\in\T^{\hspace*{1pt}r}$, represented in coordinates by
$$ T= \sum_{1\le i_1\le\dots\le i_r\le n} t_{i_1\dots i_r}e_{i_1}\cdots e_{i_r},$$
the polynomial on $\Rn$ defined by
\begin{equation}\label{5.1} 
p_T(y)=  \sum_{1\le i_1\le\dots\le i_r\le n} t_{i_1\dots i_r}y_{i_1}\cdots y_{i_r}, \qquad y=\sum_{i=1}^n y_ie_i.
\end{equation}
The mapping $T\mapsto p_T$ is a vector space isomorphism  between $\T^{\hspace*{1pt}r}$ and the vector space of homogeneous polynomials of degree $r$ on $\Rn$. It is compatible with the operation of the orthogonal group, that is, it satisfies $p_{\vartheta T}(y)= p_T(\vartheta^{-1}y)$ for $y\in\Rn$ and $\vartheta\in{\rm O}(n)$. 

We choose the orthonormal basis $(e_1,\dots,e_n)$ in such a way that $e_1,\dots,e_k$ span the subspace $L$ and $e_{k+1},\dots,e_n$ span its orthogonal complement $L^\perp$.

(a) Assume that the assumption of (a) is satisfied.  Then the polynomial $p_T$ defined by (\ref{5.1}) satisfies $p_T(\vartheta^{-1}y) = p_{\vartheta T}(y) =p_T(y)$ for each $\vartheta\in {\rm SO}(n)$ fixing $L^\perp$ pointwise. For $\rho>0$ and $\zeta_{k+1},\dots,\zeta_n\in\R$, the group  of such rotations is transitive on the set
$$ \left\{y=y_1e_1+\dots+y_ne_n\in\Rn: y_1^2+\dots+y_k^2=\rho^2,\, y_{k+1}=\zeta_{k+1},\dots,y_n=\zeta_n\right\}.$$
(Here it is used that $k\ge 2$.) Therefore, the proof of Lemma 3 in \cite{Sch13} yields the assertion.

(b) Now let $k=1$. Then we can assume that $e_1=v_L$ and write
\begin{eqnarray*} 
p_T(y) &=& \sum_{1 \le i_1\le\dots\le i_r\le n} t_{i_1\dots i_r}y_{i_1}\cdots y_{i_r}\\
&=& \sum_{j=0}^r y_1^j \sum_{2\le i_{j+1}\le\dots\le i_r\le n} t_{1\dots 1  i_{j+1}\dots i_r}y_{i_{j+1}}\cdots y_{i_r}.
\end{eqnarray*}
We define a tensor $T^{(r-j)}\in\T^{r-j}(L^\perp)$ by
$$ T^{(r-j)} := \sum_{2\le i_{j+1}\le\dots\le i_r\le n} t_{1\dots 1i_{j+1}\dots i_r}e_{i_{j+1}}\cdots e_{i_r}$$
and obtain the assertion (b).
\end{proof}

For a $\T^{\hspace*{1pt}r}$-valued Borel measure $F$ on $\bS^{n-1}$ we say that it {\em intertwines orthogonal transformations} if $F(\theta B)= (\theta F)(B)$ for all $B\in{\cal B}(\bS^{n-1})$ and all orthogonal transformations $\theta\in{\rm O}(n)$. We say that $F$ {\em intertwines rotations} if $F(\vartheta B)= (\vartheta F)(B)$ for all $B\in{\cal B}(\bS^{n-1})$ and all rotations $\vartheta\in{\rm SO}(n)$. (Note that this terminology differs from the one in \cite{Sch13}.)

We recall Lemma 4 from \cite{Sch13}.

\begin{lemma}\label{L4a}
Let $n\in{\mathbb N}$. Let $r\in{\mathbb N}_0$, and let $F:{\cal B}(\bS^{n-1})\to\T^{\hspace*{1pt}r}$ be a $\T^{\hspace*{1pt}r}$-valued measure that intertwines orthogonal transformations. Then
\begin{equation}\label{2aa} 
F(B) = \sum_{j=0}^{\lfloor r/2  \rfloor} a_jQ^j\int_B u^{r-2j}\,\cH^{n-1}(\D u)
\end{equation}
for $B\in{\cal B}(\bS^{n-1})$, with real constants $a_j$, $j=0,\dots,\lfloor r/2  \rfloor$.
\end{lemma}

In \cite{Sch13}, this lemma was proved for $n\ge 2$. If $n=1$, then ${\mathbb S}^1= \{e,-e\}$ and $F(\{e\}) =a(e)e^r$, $F(\{-e\}) =a(-e)e^r$ with real constants $a(e),a(-e)$. If $\theta\in {\rm O}(1)$ satisfies $\theta e=-e$, then $F(\theta\{e\})=\theta F(\{e\})$ yields $a(-e)=(-1)^ra(e)$. Trivially, $F$ can also be represented in the form
(\ref{2aa}) (which we prefer as a unified expression).

The following lemma concerns rotations only.

\begin{lemma}\label{L4}
Let $r\in{\mathbb N}_0$, and let $F:{\cal B}(\bS^{n-1})\to\T^{\hspace*{1pt}r}$ be a $\T^{\hspace*{1pt}r}$-valued measure that intertwines rotations.

\noindent $\rm (a)$ If $n\ge 3$, then
\begin{equation}\label{2a} 
F(B) = \sum_{j=0}^{\lfloor r/2  \rfloor} a_jQ^j\int_B u^{r-2j}\,\cH^{n-1}(\D u)
\end{equation}
for $B\in{\cal B}(\bS^{n-1})$, with real constants $a_j$, $j=0,\dots,\lfloor r/2  \rfloor$.

\noindent $\rm (b)$ Let $n=2$. Fix an orientation of $\R^2$, and for $u\in\bS^1$, let $\overline u$ be the unit vector such that $(u,\overline u)$ is a positively oriented orthonormal basis of $\R^2$. Then
$$ F(B) = \sum_{j=0}^r a_j\int_B \overline u^j  u^{r-j}\,\cH^1(\D u) $$
for $B\in{\cal B}(\bS^1)$, with real constants $a_j$. 
\end{lemma}

\begin{proof} We modify the argumentation in the proof of \cite[Lemma 4]{Sch13}, replacing the group ${\rm O}(n)$ by ${\rm SO}(n)$. We fix a vector $u\in \bS^{n-1}$ and denote by $B_{u,\rho}$ the spherical cap with centre $u$ and spherical radius $\rho\in (0,\pi/2)$. Let $T:= F(B_{u,\rho})$. Then $\vartheta T=T$ for all rotations $\vartheta\in{\rm SO}(n)$ fixing $u$. We choose the orthonormal basis $(e_1,\dots,e_n)$ such that $e_n=u$. 

(a) If $n\ge 3$, then $\dim u^\perp\ge 2$. Therefore, the proof of \cite[Lemma 4]{Sch13} goes through if we apply 
in $L=u^\perp$ part (a) of the present Lemma \ref{L3} (instead of \cite[Lemma 3]{Sch13}).

(b) Now we assume that $n=2$. We apply  Lemma \ref{L3}(b) with $L=u^\perp$ and $v_L=\overline u$. This gives
$$ T= \sum_{j=0}^r \overline u^j \pi^*_{L^\perp} T^{(r-j)}$$
with tensors $T^{(r-j)}\in \T^{r-j}({\rm lin}\{u\})$. Since every tensor in $\T^{r-j}({\rm lin}\{u\})$ is of the form $b_ju^{r-j}$ with a real constant $b_j$ (and since the tensor $T^{(r-j)}$ depends on $u$ and $\rho$), we obtain that
$$ F(B_{u,\rho})=\sum_{j=0}^r b_j(u,\rho) \overline u^j u^{r-j}.$$
This holds for all $u\in\s^1$ and does not depend on the choice of the basis. Since $F$ intertwines rotations, we have $\vartheta F(B_{u,\rho})=F(\vartheta B_{u,\rho}) = F(B_{\vartheta u,\rho})$ for $\vartheta \in{\rm SO}(2)$. This can be written as
$$ \sum_{j=0}^r b_j(u,\rho) (\vartheta \overline u)^j(\vartheta u)^{r-j} =\sum_{j=0}^r b_j(\vartheta u,\rho) (\vartheta \overline u)^j(\vartheta u)^{r-j}.$$
The tensors $(\vartheta \overline u)^j(\vartheta u)^{r-j}$, $j=0,\dots,r$, are linearly independent, hence $b_j(u,\rho)=:b_j(\rho)$ does not depend on $u$.

For given $u\in {\mathbb S}^1$ we can choose $e_2=u$ and then obtain, for $m\in \{0,\dots,r\}$, 
$$ \binom{r}{m}F(B_{u,\rho})(\underbrace{-e_1,\dots,-e_1}_m,\underbrace{e_2,\dots,e_2}_{r-m})=b_m(\rho).$$
Now we have all the ingredients to finish the proof in the same way as \cite[Lemma 4]{Sch13} was proved.
\end{proof}

\section{Proof of Theorem \ref{Thm3}}\label{sec6}

To prove Theorem \ref{Thm3}, we assume that $\Gamma: \Pn\times\B(\Sigma^n)\to \T^p$ is a mapping which has the following properties.\\[1mm]
(a) $\Gamma(P,\cdot)$ is a $\T^p$-valued measure, for each $P\in\Pn$,\\[1mm]
(b) $\Gamma$ is translation invariant and ${\rm SO}(n)$ covariant,\\[1mm]
(c) $\Gamma$ is locally defined.

We shall reduce the proof of Theorem \ref{Thm3} to the classification of a simpler type of tensor-valued mappings. Let $k\in \{0,\dots,n-1\}$, and let $L\in G(n,k)$. Let $A\in \B_b(L)$ and $B\in\B(\Sn)$. Let $P\subset L$ be a polytope with $A\subset P$. Then $A\times (B\cap L^\perp)\subset {\rm Nor}\,P$, and since $\Gamma$ is locally defined, $\Gamma(P,A\times (B\cap L^\perp))=:\varphi(A,B)$ does not depend on $P$. Since each coordinate of $\varphi(\cdot,B)$ with respect to some basis is a locally finite Borel measure which is invariant under translations of $L$ into itself, it follows that $\varphi(A,B) = \Ha^k(A)\Delta_k(L,B)$ with a tensor $\Delta_k(L,B)$. This defines a mapping 
$$\Delta_k: G(n,k)\times\B(\Sn)\to \T^p.$$ 
From the properties of $\Gamma$ it follows that this mapping has the following properties.\\[1mm]
(${\rm a}'$) $\Delta_k(L,\cdot)$ is a $\T^p$-valued measure, for each $L\in G(n,k)$,\\[1mm]
(${\rm b}'$)  $\Delta_k$ is ${\rm SO}(n)$ covariant, in the sense that
\begin{equation}\label{2.6a}
\Delta_k(\vartheta L,\vartheta B)=\vartheta\Delta_k(L,B) \quad\mbox{for }  \vartheta\in{\rm SO}(n).
\end{equation}
(${\rm c}'$) $\Delta_k(L,B)= \Delta_k(L,B\cap L^\perp)$ for $L\in G(n,k)$ and $B\in\B(\Sn)$.

Now let $P\in\Pn$, $A\in\B(\R^n)$ and $B\in\B(\Sn)$. Since $\Gamma(P,\cdot)$ is concentrated on ${\rm Nor}\,P$ (see \cite[Lemma 3.3]{HS14}, whose proof does not use ${\rm O}(n)$ covariance) and
$$ (A\times B)\cap{\rm Nor}\,P= \bigcup_{k=0}^{n-1}\bigcup_{F\in\F_k(P)} (A\cap{\rm relint}\,F)\times (B\cap\nu(P,F))$$
is a disjoint union, we have
\begin{eqnarray}\label{4}
\Gamma(P,A\times B) &=&\Gamma(P,(A\times B)\cap{\rm Nor}\,P)\nonumber\\
&=& \sum_{k=0}^{n-1} \sum_{F\in\F_k(P)} \Gamma(P,(A\cap{\rm relint}\,F)\times (B\cap\nu(P,F)))\nonumber\\
&=& \sum_{k=0}^{n-1} \sum_{F\in\F_k(P)} \Ha^k(A\cap{\rm relint}\,F)\Delta_k(L(F),B\cap\nu(P,F)).
\end{eqnarray}
This is the reason why we want to determine $\Delta_k(L,B)$.

In order to classify these mappings, let $k\in\{0,\dots,n-1\}$, $L\in G(n,k)$, and  $B \in\B(\Sn)$.

\vspace{2mm}

\noindent{\bf Case 1:} $k=0$. Then $L^\perp=\R^n$, and $\Delta_0(\{0\},\cdot):\B(\Sn)\to\T^p$ is a $\T^p$-valued measure which, by (\ref{2.6a}), intertwines rotations.  

\vspace{2mm}

\noindent{\bf Subcase 1a:} $n\ge 3$. Lemma \ref{L4}(a) gives
\begin{equation}\label{2}
\Delta_0(\{0\},B)=  \sum_{j=0}^{\lfloor p/2\rfloor} a_jQ^j\int_B u^{p-2j}\,\Ha^{n-1}(\D u)
\end{equation}
with real constants $a_j$.

\vspace{2mm}

\noindent{\bf Subcase 1b:} $n=2$. Lemma \ref{L4}(b) gives
\begin{equation}\label{3}
\Delta_0(\{0\},B)=  \sum_{j=0}^p a_j\int_B \overline u^j u^{p-j}\,\Ha^{1}(\D u)
\end{equation}
with real constants $a_j$.

\vspace{2mm}

\noindent{\bf Case 2:} $k\ge 2$. If $\vartheta\in{\rm SO}(n)$ fixes $L^\perp$ pointwise, then $\vartheta L=L$, and  it follows from (\ref{2.6a}) (together with (${\rm c}'$)) that $T:=\Delta_k(L,B)$ satisfies $\vartheta T =T$. Therefore, we infer from Lemma \ref{L3}(a) that
\begin{equation}\label{4.7}
\Delta_k(L,B) = \sum_{j=0}^{\lfloor p/2 \rfloor}Q_L^j\pi_{L^\perp}^*T^{(p-2j)}(L,B)
\end{equation}
with tensors $T^{(p-2j)}(L,B) \in\T^{p-2j}(L^\perp)$, $j=0,\dots,\lfloor\, p/2\rfloor$. 

Let $y\in L\cap \bS^{n-1}$ and $x_1,\dots,x_p\in L^\perp$. For $ q\in\{0,\dots,\lfloor p/2 \rfloor\}$, we apply both sides of (\ref{4.7}) to the $p$-tuple $(y,\dots,y,x_1,\dots,x_{p-2q})$ and obtain
\begin{equation}\label{4.7a}
\Delta_k(L,B)(y,\dots,y,x_1,\dots,x_{p-2q}) = \binom{p}{2q}^{-1}T^{(p-2q)}(L,B)(x_1,\dots,x_{p-2q}).
\end{equation}

Let $\vartheta\in {\rm SO}(n)$ and $B\in{\mathcal B}(\bS_{L^\perp})$. Let $y\in L\cap \s^{n-1}$, $x_1,\ldots,x_{p-2j}\in L^\perp$, and $j\in \{0,\ldots,\lfloor\frac{p}{2}\rfloor\}$. Then, using \eqref{4.7a}, \eqref{2.6a}, the definition of the operation of $\vartheta$ on tensors and then again \eqref{4.7a}, we get
\begin{align}\label{strictrel0}
&T^{(p-2j)}(\vartheta L,\vartheta B)(\vartheta x_1,\ldots,\vartheta x_{p-2j})\nonumber\\
&=\binom{p}{2j}\Delta_k(\vartheta L,\vartheta B)(\vartheta y,\ldots,\vartheta y,\vartheta x_1,\ldots,\vartheta x_{p-2j})\nonumber\\
&=\binom{p}{2j}[\vartheta\Delta_k(  L, B)](\vartheta y,\ldots,\vartheta y,\vartheta x_1,\ldots,\vartheta x_{p-2j})\nonumber\\
&=\binom{p}{2j} \Delta_k(  L, B)(  y,\ldots,  y,  x_1,\ldots,  x_{p-2j})\nonumber\\
&=T^{(p-2j)}( L, B)( x_1,\ldots, x_{p-2j}).
\end{align}
Let $i_L:L\to\Rn$ be the inclusion map. Later, we have to observe that 
$$
i_{\vartheta L^\perp}^*\vartheta\pi_{L^\perp}^* Q_{(L^\perp)}=i_{\vartheta L^\perp}^*\vartheta Q_{L^\perp}
=i_{\vartheta L^\perp}^* Q_{\vartheta L^\perp}\\
=Q_{(\vartheta L^\perp)}.
$$
Since $\vartheta x_i\in\vartheta L^\perp$ for $i\in\{1,\ldots,p-2j\}$, we have
\begin{align*}
&[i_{\vartheta L^\perp}^*\vartheta\pi_{L^\perp}^* T^{(p-2j)}(L,B)](\vartheta x_1,\ldots,\vartheta x_{p-2j})\\
&=[\vartheta\pi_{L^\perp}^* T^{(p-2j)}(L,B)](\vartheta x_1,\ldots,\vartheta x_{p-2j})\\
&=[ \pi_{L^\perp}^* T^{(p-2j)}(L,B)](  x_1,\ldots,  x_{p-2j})\\
&=T^{(p-2j)}(L,B)(  x_1,\ldots,  x_{p-2j}).
\end{align*}
Thus, we finally get
\begin{equation}\label{4.8a}
T^{(p-2j)}(\vartheta L,\vartheta B)=i_{\vartheta L^\perp}^*\vartheta\pi_{L^\perp}^* T^{(p-2j)}(L,B),
\end{equation}
where both sides are considered as tensors in $\mathbb{T}^{p-2j}(\vartheta L^\perp)$. (Of course, the effect
of $i_{\vartheta L^\perp}^*$ and $\pi_{L^\perp}^*$ on the right side of \eqref{4.8a} is trivial
if the appropriate domain is considered in each case.)

Let $\theta \in {\rm O}(L^\perp)$ (the orthogonal group of $L^\perp$). We can choose a rotation 
$\vartheta\in {\rm SO}(n)$ such that the restriction of $\vartheta$ to $L^\perp$
coincides with $\theta$ and $\vartheta L=L$. Then \eqref{strictrel0} (or \eqref{4.8a}) implies that
\begin{equation}\label{4.8c}
T^{(p-2j)}( L,\theta B)=\theta T^{(p-2j)}(L,B),
\end{equation}
were again both sides are considered as tensors in $\mathbb{T}^{p-2j}(L^\perp)$. 

Because of (\ref{4.8c}), it follows from Lemma \ref{L4a} (applied in $L^\perp$) that 
\begin{equation}\label{4.8}
T^{(p-2j)}(L,B) = \sum_{i=0}^{\lfloor p/2\rfloor-j} \alpha_{ipj}(L) Q_{(L^\perp)}^i \int_B u^{p-2j-2i}\,\cH^{n-k-1}(\D u)
\end{equation}
with real constants $\alpha_{ipj}(L)$ (recall that $B\in\B(\s_{L^\perp})$). 

To show that the coefficients $\alpha_{ipj}(L)$ in (\ref{4.8}) are independent of $L$, we fix a $k$-dimensional linear subspace $L_0$ and put $\alpha_{ipj}(L_0)=:\alpha_{kipj}$. For a given $k$-dimensional subspace $L$, there is a rotation $\vartheta\in{\rm SO}(n)$ with $L=\vartheta L_0$. From
(\ref{4.8a}) and (\ref{4.8}) we obtain, for $B\in{\mathcal B}(\bS_{L^\perp})$ and $B_0=\vartheta^{-1}B\in \mathcal{B}(\bS_{L_0^\perp})$,
\begin{eqnarray}\label{4.8b}
T^{(p-2j)}(L,B) &=& T^{(p-2j)}(\vartheta L_0,\vartheta B_0)= i^*_{\vartheta L_0^\perp}\vartheta 
\pi_{L^\perp_0}^* T^{(p-2j)}(L_0,B_0) \nonumber\\
&=& i^*_{\vartheta L_0^\perp}\vartheta \sum_{i=0}^{\lfloor p/2\rfloor-j} \alpha_{ipj}(L_0) Q_{L_0^\perp}^i \int_{B_0} u^{p-2j-2i}\,\cH^{n-k-1}(\D u)\nonumber\\
&=& \sum_{i=0}^{\lfloor p/2\rfloor-j} \alpha_{kipj} Q_{(\vartheta L_0^\perp)}^i \int_{\vartheta B_0} u^{p-2j-2i}\, \cH^{n-k-1}(\D u)\nonumber\\
& = &\sum_{i=0}^{\lfloor p/2\rfloor-j} \alpha_{kipj} Q_{(L^\perp)}^i \int_B u^{p-2j-2i}\,\cH^{n-k-1}(\D u).
\end{eqnarray}
Relations (\ref{4.7}) and (\ref{4.8b}) now yield
$$ \Delta_k(L,B) = \sum_{j=0}^{\lfloor p/2 \rfloor} Q_L^j \sum_{i=0}^{\lfloor p/2 \rfloor-j} \alpha_{kipj}  Q_{L^\perp}^i \int_B u^{p-2j-2i}\,\cH^{n-k-1}(\D u).$$
Inserting $Q_{L^\perp}=Q-Q_L$, expanding and regrouping, we see that
\begin{equation}\label{16''}
\Delta_k(L,B)=\sum_{a=0}^{\lfloor p/2\rfloor} \sum_{b=a}^{\lfloor p/2\rfloor}  c_{pkab} Q^aQ_L^{b-a} \int_B u^{p-2b}\,\cH^{n-k-1}(\D u)
\end{equation}
with real constants $c_{pkab}$.

\vspace{2mm}

\noindent{\bf Case 3:} $k=1$. Again, we assume that $B\in\B(\bS_{L^\perp})$. Instead of (\ref{4.7}), we can only infer from Lemma \ref{L3}(b) that, after choosing a unit vector $v_L$ spanning $L$, we have
\begin{equation}\label{16'}
\Delta_1(L,B) = \sum_{j=0}^p v_L^j\pi_{L^\perp}^*T^{(p-j)}(L,B)
\end{equation}
with tensors $T^{(p-j)}(L,B)\in \T^{p-j}(L^\perp)$, $j=0,\dots, p$. Let $x_1,\dots,x_p\in L^\perp$. For $q\in\{0,\dots,p\}$, we apply both sides of (\ref{16'}) to the $p$-tuple 
\begin{equation}\label{18'}
(\underbrace{v_L,\dots,v_L}_q,x_1,\dots,x_{p-q})\end{equation} 
and obtain
\begin{equation}\label{17'}
\Delta_1(L,B)(v_L,\dots,v_L,x_1,\dots,x_{p-q}) = \binom{p}{q}^{-1}T^{(p-q)}(L,B)(x_1,\dots,x_{p-q}).
\end{equation}
Again, $T^{(p-q)}(L,B)$ is a $\T^{p-q}(L^\perp)$-valued measure on $\s_{L^\perp}$. It intertwines rotations of $L^\perp$.

\vspace{2mm}

\noindent{\bf Subcase 3a:} $n\ge 4$. Then $\dim L^\perp\ge 3$. Hence, we can apply Lemma \ref{L4}(a) in $L^\perp$ and obtain that
\begin{equation}\label{18} 
T^{(p-q)}(L,B)=\sum_{i=0}^{\lfloor \frac{p-q}{2}\rfloor} \beta_{pqi}(L) Q_{(L^\perp)}^i \int_B u^{p-q-2i}\,\cH^{n-2}(\D u).
\end{equation}
In the same way as (\ref{4.8b}) was deduced, we conclude that
\begin{equation}\label{18a} 
T^{(p-q)}(L,B) =\sum_{i=0}^{\lfloor \frac{p-q}{2}\rfloor} \beta_{pqi} Q_{(L^\perp)}^i \int_B u^{p-q-2i}\,\cH^{n-2}(\D u)
\end{equation}
with constants $\beta_{pqi}$. Relations (\ref{16'}) and (\ref{18a}) yield 
\begin{equation}\label{19}
\Delta_1(L,B) =\sum_{j=0}^p v_L^j \sum_{i=0}^{\lfloor \frac{p-j}{2}\rfloor} \beta_{pji} Q_{L^\perp}^i \int_B u^{p-j-2i}\,\cH^{n-2}(\D u).
\end{equation}
Since $v_L^2=Q_L$, we distinguish whether $j$ is even or odd and write (\ref{19}) as
$$ \Delta_1(L,B) = \Delta_1^{(0)}(L,B)+\Delta_1^{(1)}(L,B)$$
with 
\begin{eqnarray} \label{34a}
\Delta_1^{(0)}(L,B) &=& \sum_{a=0}^{\lfloor p/2\rfloor} Q_L^a \sum_{i=0}^{\lfloor p/2\rfloor-a} \beta_{p(2a)i}Q_{L^\perp}^i \int_B u^{p-2a-2i}\,\cH^{n-2}(\D u),\nonumber\\
\Delta_1^{(1)}(L,B) &=&\sum_{b=0}^{\lfloor\frac{p-1}{2}\rfloor} Q_L^b v_L\sum_{i=0}^{\lfloor \frac{p-1}{2} \rfloor-b} \beta_{p(2b+1)i} Q_{L^\perp}^i \int_B u^{p-2b-1-2i}\,\cH^{n-2}(\D u).
\end{eqnarray}
We can choose a rotation $\vartheta\in{\rm SO}(n)$ such that $\vartheta v_L=-v_L$ and that the restriction of $\vartheta$ to $L^\perp$ is a reflection of $L^\perp$ into itself. Moreover, we specialize $B$ to $B'$ such that $\vartheta B'=B'$. Then the last equation yields $\vartheta\Delta_1^{(1)}(L,B')=-\Delta_1^{(1)}(L,B')$, whereas the rotation covariance of $\Delta_1$ and of $\Delta_1^{(0)}$ yields $\vartheta\Delta_1^{(1)}(L,B')=\Delta_1^{(1)}(L,B')$. Thus, we obtain $\Delta_1^{(1)}(L,B')=0$ for all $B'\in{\mathcal B}({\mathbb S}_{L^\perp})$ with $\vartheta B'=B'$. Inserting (\ref{18'}), with various $q$, into (\ref{34a}) for $B'$ (for which it is zero), we deduce that
$$ \sum_{i=0}^{\lfloor \frac{p-1}{2}\rfloor-b} \beta_{p(2b+1)i}  Q_{L^\perp}^i \int_{B'} u^{p-2b-1-2i}\,\cH^{n-2}(\D u)=0$$
for $b=0,\dots,\lfloor \frac{p-1}{2}\rfloor$. Here $B'$ can be any Borel set in $\s_{L^\perp}$ which is invariant under some reflection of $\s_{L^\perp}$. Therefore, we can deduce that
$$ \sum_{i=0}^{\lfloor \frac{p-1}{2}\rfloor-b} \beta_{p(2b+1)i}  Q_{L^\perp}^i u^{p-2b-1-2i}=0$$
for all $u\in\s_{L^\perp}$. As in the proof of Proposition \ref{Lem1}, we conclude that all coefficients $\beta_{p(2b+1)i}$ are zero. It follows that $\Delta_1^{(1)}(L,B)=0$ for all $B$ and therefore $\Delta_1(L,B)= \Delta_1^{(0)}(L,B)$.
Since $Q_{L^\perp}=Q - Q_L$, we obtain
\begin{equation}\label{4a}
\Delta_1(L,B)=\sum_{a=0}^{\lfloor p/2\rfloor} \sum_{b=a}^{\lfloor p/2\rfloor}  c_{p1ab} Q^aQ_L^{b-a} \int_B u^{p-2b}\,\cH^{n-2}(\D u)
\end{equation}
with real constants $c_{p1ab}$.

\vspace{2mm}

\noindent{\bf Subcase 3b:} $n=3$. The choice of the unit vector $v_L\in L$ in Case 3 determines (together with the given orientation of $\R^3$) an orientation of $L^\perp$. For a given unit vector $u\in L^\perp$, let $\overline u\in L^\perp$ be the unique unit vector such that $(v_L,u,\overline u)$ is an orthonormal basis of $\R^3$. In other words, $\overline u= v_L\times u$, where $\times$ means the vector product.

Lemma \ref{L4}(b), applied in $L^\perp$, yields
$$ T^{(p-q)}(L,B) = \sum_{i=0}^{p-q} a_i(L) \int_B \overline u^i u^{p-q-i}\,\cH^1(\D u)$$
with constants $a_i(L)$. Arguments as used previously in Case 2 show that $a_i(L)=a_i$ is independent of $L$. With this and (\ref{16'}) we get 
\begin{equation}\label{30} 
\Delta_1(L,B) =  \sum_{q=0}^{p} v_L^q \sum_{i=0}^{p-q} a_i \int_B (v_L\times u)^i u^{p-q-i}\,\cH^1(\D u).
\end{equation}
We write
$$ \Delta_1 = \Delta_1^{(00)} + \Delta_1^{(10)} + \Delta_1^{(01)} + \Delta_1^{(11)},$$
where $\Delta_1^{(10)}=\Delta_1^{(01)}=0$ if $p=0$, $\Delta_1^{(11)}=0$ if $p\le 1$, and otherwise
$$ \Delta_1^{(\alpha\beta)}(L,B):=  \sum_{q=0 \atop q\equiv \alpha\,{\rm mod}\,2}^p v_L^q \sum_{i=0 \atop i\equiv \beta\,{\rm mod}\,2}^{p-q} a_i \int_B (v_L\times u)^i u^{p-q-i}\,\cH^1(\D u)$$
for $\alpha,\beta\in\{0,1\}$. Using that $v_L^2=Q_L$, $(v_L\times u)^2=Q_{L^\perp}-u^2$ and $Q=Q_L+Q_{L^\perp}$, we get 
$$\Delta_1^{(00)}(L,B) = \sum_{j=0}^{\lfloor p/2\rfloor} Q_L^j \sum_{m=0}^{\lfloor p/2\rfloor-j} b_m \int_B(Q-Q_L-u^2)^m u^{p-2j-2m}\,\cH^1(\D u).$$
After expanding and regrouping, this can be written as
\begin{equation}\label{31}
\Delta_1^{(00)}(L,B) =\sum_{a=0}^{\lfloor p/2\rfloor} \sum_{b=a}^{\lfloor p/2\rfloor} a_{pab} Q^a  Q_L^{b-a}  \int_B  u^{p-2b}\,\cH^1(\D u).
\end{equation}
In the same way, we obtain the representations
\begin{eqnarray*}
\Delta_1^{(10)}(L,B) &=&\sum_{a=0}^{\lfloor \frac{p-1}{2}\rfloor} \sum_{b=a}^{\lfloor \frac{p-1}{2}\rfloor} b_{pab} Q^a  Q_L^{b-a} v_L \int_B  u^{p-2b-1}\,\cH^1(\D u), \quad\mbox{if }p\ge 1,\\
\Delta_1^{(01)}(L,B) &=& \sum_{a=0}^{\lfloor \frac{p-1}{2}\rfloor} \sum_{b=a}^{\lfloor \frac{p-1}{2}\rfloor} c_{pab} Q^a Q_L^{b-a}  \int_B (v_L\times u) u^{p-2b-1}\,\cH^1(\D u),\quad\mbox{if }p\ge 1,\\
\Delta_1^{(11)}(L,B) &=& \sum_{a=0}^{\lfloor \frac{p-2}{2}\rfloor} \sum_{b=a}^{\lfloor \frac{p-2}{2}\rfloor} d_{pab} Q^a Q_L^{b-a} v_L \int_B (v_L\times u) u^{p-2b-2}\,\cH^1(\D u),\quad\mbox{if }p\ge 2.
\end{eqnarray*}
Arguing as in Subcase 3a, we can show that $\Delta_1^{(10)}(L,B)+\Delta_1^{(01)}(L,B)=0$. It follows that 
\begin{eqnarray}\label{32}
\Delta_1(L,B) &=& \sum_{a=0}^{\lfloor \frac{p}{2}\rfloor} \sum_{b=a}^{\lfloor \frac{p}{2}\rfloor} a_{pab} Q^a  Q_L^{b-a} \int_B  u^{p-2b}\,\cH^1(\D u)\nonumber\\
&& + \sum_{a=0}^{\lfloor \frac{p}{2}\rfloor-1} \sum_{b=a}^{\lfloor \frac{p}{2}\rfloor-1} d_{pab} Q^a  Q_L^{b-a} v_L \int_B (v_L\times u) u^{p-2b-2}\,\cH^1(\D u).
\end{eqnarray}

\vspace{2mm}

\noindent{\bf Subcase 3c:} $n=2$. By (\ref{16'}),
\begin{equation}\label{6}
\Delta_1(L,B) = \sum_{q=0}^p v_L^q\pi_{L^\perp}^*T^{(p-q)}(L,B)
\end{equation}
with $T^{(p-q)}(L,B)\in\T^{p-q}(L^\perp)$. We can assume that $v_L =\overline{u_L}$, where $u_L$ is one of the two unit normal vectors of $L$. Then $B\subset\{u_L, -u_L\}$ and 
$$ T^{(p-q)}(L,B) = c_{p-q}(L,B)u_L^{p-q}.$$
As above, we have
\begin{equation}\label{5}
T^{(p-q)}(\vartheta L,\vartheta B) =\vartheta T^{(p-q)}(L,B)\quad\mbox{for } \vartheta\in {\rm SO}(2).
\end{equation}
First suppose that $B=\{u_L\}$. Using (\ref{5}) for the rotation $\vartheta$ with $\vartheta u_L=-u_L$, we see that $c_{p-q}(L,\{-u_L\})=c_{p-q}(L,\{u_L\})=: c_{p-q}(L)$. Thus, for this $B$, we can write (\ref{6}) in the form
$$
\Delta_1(L,B) =\sum_{q=0}^p c_{p-q}(L) \int_B\overline u^qu^{p-q}\,\Ha^0(\D u).
$$ This holds also for $B=\{-u_L\}$, and since $\Delta_1(L,\{u_L,-u_L\})= \Delta_1(L,\{u_L\})+ \Delta_1(L,\{-u_L\})$, it holds for arbitrary $B\in\B(\s_{L^\perp})$. Now we can deduce as in Case 2 that $c_{p-q}(L):=c_{p-q}$ is also independent of $L$. Since $\overline u^2+u^2=Q$, we obtain
\begin{equation}\label{8a}
\Delta_1(L,B) =\sum_{a=0}^{\lfloor p/2\rfloor} \alpha_a Q^a\int_B u^{p-2a}\,\Ha^0(\D u) + \sum_{a=0}^{\lfloor \frac{p-1}{2}\rfloor} \beta_a Q^a\int_B \overline u u^{p-2a-1}\,\Ha^0(\D u).
\end{equation}

\vspace{2mm}

The representations (\ref{2}), (\ref{3}), (\ref{16''}), (\ref{4a}), (\ref{32}), (\ref{8a}) obtained for $\Delta_k$ now allow us to evaluate (\ref{4}). 

Let $P\in\Pn$, $A\in\B(\R^n)$ and $B\in\B(\Sigma^n)$. We consider first the case where $n=3$. Using (\ref{2}) for $k=0$, (\ref{32}) for $k=1$, and (\ref{16''}) for $k=2$, we can write (\ref{4}) in the form
\begin{eqnarray*}
&&\Gamma(P,A\times B)\\
&& =\sum_{F\in\F_0(P)} \Ha^0(A\cap {\rm relint}\,F) \sum_{j=0}^{\lfloor \frac{p}{2}\rfloor} a_j Q^j \int_{B\cap\nu(P,F)} u^{p-2j}\,\Ha^2(\D u)\\
&& \hspace{4mm} + \sum_{F\in\F_1(P)} \Ha^1(A\cap {\rm relint}\,F)\Bigg\{ \sum_{a=0}^{\lfloor \frac{p}{2}\rfloor} \sum_{b=a}^{\lfloor \frac{p}{2}\rfloor} a_{pab} Q^a Q_{L(F)}^{b-a}\int_{B\cap\nu(P,F)} u^{p-2b}\,\Ha^1(\D u)\\
&&  \hspace{4mm}+ \sum_{a=0}^{\lfloor \frac{p}{2}\rfloor-1} \sum_{b=a}^{\lfloor \frac{p}{2}\rfloor-1}d_{pab}Q^a Q_{L(F)}^{b-a} v_{L(F)}\int_{B\cap\nu(P,F)}(v_{L(F)}\times u) u^{p-2b-2}\,\Ha^1(\D u)\Bigg\}\\
&&  \hspace{4mm}+ \sum_{F\in\F_2(P)} \Ha^2(A\cap {\rm relint}\,F) \sum_{a=0}^{\lfloor \frac{p}{2}\rfloor} \sum_{b=a}^{\lfloor \frac{p}{2}\rfloor} c_{p2ab} Q^a Q_{L(F)}^{b-a}\int_{B\cap\nu(P,F)} u^{p-2b}\,\Ha^0(\D u).
\end{eqnarray*}
From (\ref{2.6}) and (\ref{1a}) we have
\begin{align*} 
& \sum_{f\in\F_0(P)}\Ha^0(A\cap {\rm relint}\,F)  Q^j\int_{B\cap\nu(P,F)} u^{p-2j}\,\Ha^2(\D u)={\rm const}\cdot Q^j\phi_0^{0,p-2j,0}(P,A\times B),\\
& \sum_{F\in\F_1(P)} \Ha^1(A\cap {\rm relint}\,F) Q^a Q_{L(F)}^{b-a}\int_{B\cap\nu(P,F)} u^{p-2b}\,\Ha^1(\D u)={\rm const}\cdot Q^a\phi_1^{0,p-2b,b-a}(P,A\times B), \\
& \sum_{F\in\F_1(P)}  \Ha^1(A\cap {\rm relint}\,F)Q^a Q_{L(F)}^{b-a}v_{L(F)} \int_{B\cap\nu(P,F)}(v_{L(F)}\times u) u^{p-2b-2}\,\Ha^1(\D u)\\
& = {\rm const}\cdot Q^a\widetilde\phi^{0,p-2b-2,b-a}(P,A\times B),
\end{align*}
$$ \sum_{F\in\F_2(P)} \Ha^2(A\cap {\rm relint}\,F) Q^a Q_{L(F)}^{b-a} \int_{B\cap\nu(P,F)} u^{p-2b}\,\Ha^0(\D u)={\rm const}\cdot Q^a\phi_2^{0,p-2b,b-a}(P,A\times B).$$
This shows that $\Gamma(P,A\times B)$ is a linear combination of expressions $Q^m\phi_k^{0,s,j}(P,A\times B)$ and $Q^m\widetilde\phi^{0,s,j}(P,A\times B)$, with coefficients independent of $P,A,B$ and with indices as specified in Theorem \ref{Thm3}. Since $\Gamma(P,\cdot)$, $Q^m\phi_k^{0,s,j}(P,\cdot)$, $Q^m\widetilde\phi^{0,s,j}(P,\cdot)$ are measures on $\B(\Sigma^n)$, the relations still hold if $A\times B$ is replaced by a general $\eta\in\B(\Sigma^n)$. This proves Theorem \ref{Thm3} in the case where $n=3$.

The proof for $n\ge 4$ is analogous, using (\ref{2}) for $k=0$, (\ref{4a}) for $k=1$, and (\ref{16''}) for $k\ge 2$. Also the proof for $n=2$ is analogous, where (\ref{3}) is used for $k=0$ and (\ref{8a}) for $k=1$. This finishes the proof of Theorem \ref{Thm3}. \qed

\noindent Authors' addresses:\\[2mm]
\noindent Daniel Hug\\
Karlsruhe Institute of Technology \\
Department of Mathematics \\
D-76128 Karlsruhe, Germany\\
e-mail: daniel.hug@kit.edu\\[2mm]
\noindent Rolf Schneider\\
Albert-Ludwigs-Universit\"at\\
Mathematisches Institut\\
D-79104 Freiburg i. Br., Germany\\
e-mail: rolf.schneider@math.uni-freiburg.de

\end{document}